\documentclass[12pt,reqno]{amsart}

\usepackage{amssymb}
\usepackage{booktabs}
\usepackage[bookmarks=false,unicode,colorlinks,urlcolor=red,citecolor=red,linkcolor=blue]{hyperref}
\usepackage{aliascnt}
\usepackage{tikz}
\usepackage{pgfplots} 
\usepgfplotslibrary{polar}

\usepackage{color}

\usepackage[margin=1.25in]{geometry}

\newtheorem*{definition}{Definition}

\newtheorem{theorem}{Theorem}[section]

\newaliascnt{lemma}{theorem}
\newtheorem{lemma}[lemma]{Lemma}
\aliascntresetthe{lemma}

\newaliascnt{proposition}{theorem}
\newtheorem{proposition}[proposition]{Proposition}
\aliascntresetthe{proposition}

\newaliascnt{corollary}{theorem}

\aliascntresetthe{corollary}

\newaliascnt{conjecture}{theorem}

\aliascntresetthe{conjecture}

\newaliascnt{example}{theorem}

\aliascntresetthe{example}

\makeatletter
\def\tagform@#1{\maketag@@@{\ignorespaces#1\unskip\@@italiccorr}}
\let\orgtheequation\theequation
\def\theequation{(\orgtheequation)}
\makeatother
\def\equationautorefname~{}

\newcommand{\R}{{\mathbb R}}
\newcommand{\Rd}{{\R^d}}
\newcommand{\Rdz}{{\R^{d}\setminus\{0\}}}

\newcommand{\tr}{\operatorname{tr}}

\newcommand{\E}{{\mathbb E}}
\renewcommand{\H}{{\mathbb H}}
\renewcommand{\P}{{\mathbb P}}
\newcommand{\ualpha}{{\underline{\alpha}}}
\newcommand{\oalpha}{{\overline{\alpha}}}

\newcommand{\diam}{\operatorname{diam}}
\newcommand{\Cf}{C_{1}}

\newcommand{\lC}{{\underline{C}}}
\newcommand{\uC}{{\overline{C}}}
\newcommand{\la}{{\underline{\alpha}}}
\newcommand{\ua}{{\overline{\alpha}}}
\newcommand{\lt}{{\underline{\theta}}}
\newcommand{\ut}{{\overline{\theta}}}
\newcommand{\WUSC}[3]{\textrm{\rm WUSC}(#1,#2,#3)}
\newcommand{\WLSC}[3]{\textrm{\rm WLSC}(#1,#2,#3)}

\newcommand{\A}{({\bf H})}
\newcommand{\As}{({\bf H^*})}
\definecolor{bs}{RGB}{255,0,0}
\definecolor{kb}{RGB}{0,255,0}

\begin{document}

\title[Trace estimates for L\'evy processes]{Trace estimates for unimodal L\'evy processes}
\author[]{K. Bogdan and B. A. Siudeja}
\address{Wroc{\l}aw University of Technology, Wroc{\l}aw, Poland}
\email{Krzysztof.Bogdan\@@pwr.edu.pl}
\address{Department of Mathematics, Univ.\ of Oregon, Eugene,
OR 97403, U.S.A.}
\email{Siudeja\@@uoregon.edu}
\date{\today}
\thanks{The authors were partially supported by NCN grant 2012/07/B/ST1/03356.}
\keywords{Unimodal L\'evy process, weak scaling, trace asymptotics, smooth domain}
\subjclass[2010]{\text{Primary 60J75. Secondary 60J35}}

\begin{abstract}
We give two-term small-time approximation  for the trace of the Dirichlet heat kernel of bounded smooth domain for unimodal L\'evy processes satisfying the weak scaling conditions. 
\end{abstract}

\maketitle

\section{Introduction}
A two-term small-time uniform approximation for the trace of the transition density of the Wiener process killed off bounded $R$-smooth  domain  $D\subset \Rd$, i.e. the classical Dirichlet heat kernel, was obtained by van den Berg \cite{MR880981}.  
The first  term of the approximation 
is proportional to the domain's volume $|D|$ and the second--to the
surface measure $|\partial D|$ of the boundary, with explicit coefficient depending on time. 
Asymptotic non-uniform expansions of the trace of the heat kernel
were given earlier in \cite{MR0061750}, see the discussion in \cite{MR880981}.

Ba\~nuelos and Kulczycki \cite{MR2438694} obtained a uniform two-term approximation 
for
the isotropic $\alpha$-stable L\'evy processes.
The closely related case of the relativistic $\alpha$-stable L\'evy processes 
was resolved by Ba\~nuelos, Mijena and Nane \cite{MR3111870}. A similar two-term approximation for Lipschitz domains was given for the Wiener process by Brown \cite{MR1134755}, and for the isotropic $\alpha$-stable L\'evy processes--by Ba\~nuelos, Kulczycki and Siudeja \cite{MR2568694}. 
Park and Song \cite{MR3269724} obtained a two-term small-time approximation of the trace for the relativistic $\alpha$-stable L\'evy processes on Lipschitz domains, and gave an explicit power expansion of the first term.

In this work we investigate those L\'evy processes $X_t$ in $\Rd$, where $d\ge 2$, which are unimodal and satisfy the so-called weak lower and upper scaling conditions, denoted
WLSC
and WUSC respectively, 
of orders strictly between $0$ and $2$ (see \autoref{sec:prel} for details). The isotropic stable and relativistic L\'evy processes are included as special cases but at present the orders of the lower and upper scalings may differ. 
For bounded $R$-smooth open sets $D\subset \Rd$ (also called $C^{1,1}$ open sets in the literature)   our main result gives a two-term small-time approximation  of the trace of the corresponding Dirichlet heat kernel. For instance we resolve sums of independent isotropic stable L\'evy processes with different indexes.

In what follows we let 
$\psi$ be the L\'evy-Khintchine exponent and $p_t(x)$ be the transition density of $X_t$.
We consider $$\tau_D=\{t>0: X_t\not \in D\},$$
the first time that $X_t$ exits $D$.
For $t>0$ and $x,y\in \Rd$, we define the heat remainder
\begin{align}\label{eq:rdE}
  r_D(t,x,y)=\E^x\left[ \tau_D<t,p_{t-\tau_D}(X(\tau_D)-y) \right].
\end{align}
The Dirichlet heat kernel for $X_t$ is given by the Hunt formula:
\begin{align}\label{eq:wH}
 p_D(t,x,y)=p_t(y-x)-r_D(t,x,y),
\end{align}
and the trace of $X_t$ on $D$ is
\begin{equation}\label{eq:deftr}
\tr(t,D)=\int p_D(t,x,x)dx,\qquad t>0.
\end{equation}
We denote $\H=\{(x_1,\ldots,x_d)\in \Rd: x_1>0\}$, a half-space, and for $t>0$ we let
  \begin{align*}
  C_\H(t)&=\int_0^{\infty} r_\H(t,(q,0,\cdots,0),(q,0,\cdots,0))dq.
  \end{align*}
For instance, $C_\H(t)=ct^{-d/\alpha+1/\alpha}$ for the isotropic $\alpha$-stable L\'evy process \cite{MR2438694}.
Here is our main result (a stronger statement is given as \autoref{th:mainR} in \autoref{s:mt}). 
\begin{theorem}\label{mainthm}
If  bounded open set $D\subset \Rd$ is $R$-smooth, 
{\rm WLSC} and {\rm WUSC} hold for $\psi$, and $t\to 0$, then
$\tr(t,D)$ equals $p_t(0)|D|-C_\H(t)|\partial D|$ plus lower order terms.
\end{theorem}

Heuristically, if $x\in D$ and $t>0$ is small, then $r_D(t,x,x)$ is small and so $p_D(t,x,x)$ is close to $p_{\Rd}(t,x,x)=p_t(0)$. Therefore the first approximation to $\tr(t,D)$ is $p_t(0)|D|$. 
The second term in \autoref{mainthm}, $C_\H(t)|\partial D|$, approximates  $\int_D r_D(t,x,x)dx$.
As we shall see, $r_D(t,x,x)$ depends primarily on the distance of $x$ from $\partial D$. It is here that the $R$-smoothness of $D$ plays a role by allowing for an asymptotic coefficient independent of $D$, that is $C_\H(t)$. 
In view of the definition of $C_\H(t)$, the appearance of $|\partial D|$ in the second term of the approximation of the trace is natural.

In some cases, including the relativistic stable L\'evy process, explicit expansions of $p_t(0)$ can be given \cite[Lemma 3.2]{MR3269724}. In more general situations $p_t(0)$, $C_\H(t)$ and the bounds for the error terms cannot be entirely explicit but 
\autoref{lem:lbCH} and \autoref{th:mainR} below provide a satisfactory formulation.

Technically we only need to estimate $\int_D r_D(t,x,x)dx$ to prove Theorem~\eqref{mainthm}. In this connection we note that sharp global estimates for $p_D(t,x,y)$  were recently obtained by Bogdan, Grzywny and Ryznar \cite{MR3249349}, but
these estimates  do not easily translate into sharp estimates of $r_D(t,x,y)$. Namely, if $p_D(t,x,y)$ is only known to be proportional to $p_t(y-x)$, then 
essential further work is needed to accurately estimate $r_D(t,x,y)$.

The paper is composed as follows. In \autoref{sec:prel} we give preliminaries on unimodal L\'evy processes with scaling, their heat kernel, Green function and Poisson kernel for $R$-smooth open sets. 
In \autoref{s:mt} we prove \autoref{th:mainR}, a stronger and more detailed variant of \autoref{mainthm}.
The most technical step of the proof of \autoref{th:mainR}
is given separately in \autoref{proofprop}.

We remark in passing that the trace can also be studied and interpreted within the spectral theory of the corresponding semigroup
given by the integral kernel $p_D$ \cite{MR2438694}.
In view toward further research we note that 
sharp 
pointwise estimates of $r_D(t,x,y)$ complementing \cite{MR3249349}
would be of considerable interest.
We also note that two-term approximations of the trace of the heat kernel of general unimodal L\'evy processes  are open for Lipschitz domains.

{\bf Acknowledgments.} We thank Tomasz Grzywny for very helpful discussions and suggestions on the manuscript.
\section{Preliminaries}\label{sec:prel}

\subsection{Unimodality}\label{sec:uni}
A Borel measure 
on $\Rd$ is called isotropic unimodal, in short: unimodal, if
on $\Rd\setminus \{0\}$ it is absolutely continuous with respect to the Lebesgue measure
and has a  radially nonincreasing, in particular rotationally invariant, or isotropic density function. 
Recall that L\'evy measure is
an arbitrary Borel measure concentrated on $\Rdz$ and such that
\begin{equation*}\label{wml}
\int_\Rd \left(|x|^2\wedge 1\right)\nu(dx)<\infty.
\end{equation*}
In what follows we assume that $\nu$ is a
unimodal L\'evy measure
and define
\begin{equation}\label{eq:LKe}
 \psi(\xi)=\int_\Rd \left(1- \cos \left<\xi,x\right>\right) \nu(dx),\qquad\xi\in\Rd,
\end{equation}
the L\'evy-Khintchine exponent. 
It is a radial function, and we often let
$\psi(r)=\psi(\xi)$, where  $\xi\in \Rd$ and  $r=|\xi|\ge 0$. 
The same convention applies to all radial functions.
The (radially nonincreasing) density function of the unimodal L\'evy measure $\nu$ will also be denoted by $\nu$, so $\nu(dx)=\nu(x)dx$ and $\nu(x)=\nu(|x|)$.
We point out that for $\lambda\ge 1$ and $r\ge 0$, $\psi(\lambda r)\ge \pi^{-2} \psi(r)$ 
and $\psi(\lambda r)\le \pi^{-2} \lambda^2\psi(r)$
\cite[Section 4]{MR3165234}. More restrictive inequalities of this type define what are called the weak scaling conditions,
see \autoref{ss:s}.

We consider the pure-jump
L\'evy process $X=(X_t, \,t\ge 0)$ on $\Rd$ \cite{MR1739520}, in short: $X_t$,
determined by the L\'evy-Khintchine formula
$$
\E\,e^{i\left<\xi, X_t\right>}=e^{-t\psi(\xi)}=\int_\Rd e^{i\left<\xi,x\right>}p_t(dx).
$$
The process is (isotropic) unimodal, meaning that all its one-dimensional distributions
$p_t(dx)$ are (isotropic) unimodal; in fact the unimodality of $\nu$ is also necessary for the unimodality of $X_t$ \cite{MR705619}.
In what follows we always assume that $\psi$ is unbounded, equivalently that $\nu(\Rd)=\infty$.
In other words $X_t$ below is not a compound Poisson process. 
Clearly, $\psi(0)=0$ and $\psi(u)>0$ for $u>0$.
By \cite[Lemma~1.1]{MR3249349}, $p_t(dx)$ have bounded, in fact smooth density functions $p_t(x)$ for all $t>0$ if and only if 
the following Hartman-Wintner condition holds,
\begin{align}\label{eqHW}
&\lim_{|\xi|\to \infty}\psi(\xi)/\ln |\xi|=\infty.
\end{align}

Let $V$ be the renewal function of the corresponding
ladder-height process  of the first coordinate of $X_t$. Namely we consider $X_t^{(1)}$, the first coordinate process of $X_t$,
its running maximum $M_t:=\sup_{0\le s\leq t} X_s^{(1)}$ and the local time $L_t$ of $M_t-X_t^{(1)}$ at $0$ so normalized that its inverse function $L^{-1}_t$ is a standard $1/2$-stable subordinator. 
The resulting ladder-height process $\eta(t):=X^{(1)}(L^{-1}_t)$ is a subordinator with the Laplace exponent
\begin{equation*}\label{kappa}
 \kappa(u)=-\log \E e^{-u \eta(1)}=
\exp\left\{\frac{1}{\pi} \int_0^\infty \frac{ \log {\psi}(u\zeta)}{1 + \zeta^2} \, d\zeta\right\}, \quad u\ge 0,
\end{equation*}
and $V(x)$ is defined as the accumulated potential of $\eta$:
$$
V(x)=\E \int_0^\infty {\bf 1}_{[0,x]}(\eta_t)dt, \qquad x\ge 0. 
$$
For $x<0$ we let $V(x)=0$.
For instance, if $\psi(\xi)= |\xi|^\alpha$ with $\alpha\in (0,2)$, then
$V(x)= x_+^{\alpha/2}$
\cite[Example~3.7]{MR2453779}.
Silverstein studied $V$ and $V'$ as $g$ and $\psi$ in \cite[(1.8) and Theorem~2]{MR573292}.
The  Laplace transform of $V$ is
\begin{equation*}\label{eq:tLV}
\int_0^\infty V(x)e^{-u x}dx=\frac{1}{u\kappa(u)}, \qquad u>0.
\end{equation*}
The function $V$ is continuous and strictly increasing from $[0,\infty)$ onto $[0,\infty)$.
We have 
$
\lim_{r\to \infty}V(r)=\infty$. Also, $V$ is
subadditive:
\begin{equation}\label{subad}
 V(x+y)\le V(x)+V(y), \quad x,y \in \R.
\end{equation}
For a more detailed discussion of $V$ we refer the reader to
\cite{BGRptrf} and \cite{MR573292}.

In estimates we can use $V$ 
and $\psi$ 
interchangeably because by \cite[Lemma~1.2]{MR3249349},
\begin{equation}\label{cVh1pg}
V(r)
\approx \left[\psi(1/r)\right]^{-1/2}
,\qquad r>0.
\end{equation}
The above means that there is a {\it constant}, i.e. a number $C\in (0,\infty)$, such that for all $r>0$ we have
$C^{-1}V(r)\le \left[\psi(1/r)\right]^{-1/2}\le C V(r)$.
In fact in \eqref{cVh1pg} we have $C=C(d)$, meaning that $C$ may be so chosen to depend only on the dimension, see ibid.
Similar notational conventions are used throughout the paper.
To give full justice to $V$, the function 
is absolutely crucial
in the proofs of 
\cite{BGRptrf}, a paper leading to \cite{MR3249349}.
By \eqref{subad},
\begin{align}\label{eq:sublinear}
  \frac12\varepsilon V(r)\le V(\varepsilon r)\le V(r),\qquad 0<\varepsilon\le 1, \quad 0<r<\infty.
\end{align}

\subsection{Scaling}\label{ss:s}
We shall assume  relative power-type behaviors of $\psi(r)$
at infinity.
Namely we say that
$\psi$ satisfies the weak lower scaling condition
at infinity (WLSC) if there are numbers
$\la>0$, $\lt\in [0,\infty)$
and  $\lC\in(0,1]$,  such that
\begin{equation*}\label{eq:LSC2}
 \psi(\lambda r )\ge
\lC\lambda^{\,\la} \psi( r )\quad \mbox{for}\quad \lambda\ge 1, \quad
 r >\lt.
\end{equation*}
Put differently and more explicitly, $
\psi( r )/ r^{\, \la}$ is almost increasing on $(\lt,\infty)$, i.e.
\begin{equation*}\label{eq:LSC2a}
 \frac{\psi(s)}{s^{\,\la}}\ge
\lC\frac{\psi( r )}{ r ^{\, \la}},\quad \mbox{ if } \quad\; s\ge  r >\lt.
\end{equation*}
In short we write $\psi\in\WLSC{\la}{ \lt}{\lC}$, $\psi\in {\rm WLSC}(\la,\lt)$, $\psi\in {\rm WLSC}(\la)$ or  $\psi\in {\rm WLSC}$, depending on how specific we wish to be about the constants.
If $\psi\in$WLSC$(\la,\lt)$, then we say
that $\psi$ satisfies the {\it global}
weak lower scaling condition (global WLSC) if $\lt=0$.
If $\lt\ge 0$, then we can emphasize this by calling the scaling {\it local} at infinity.
We always assume that $\psi\not \equiv 0$, therefore in view of $\psi\in$WLSC
we have
the Hartman-Wintner condition \eqref{eqHW} satisfied, and so $\Rd\ni x\mapsto p_t(x)$ is smooth for each $t>0$.

Similarly,
the weak upper scaling condition
at infinity (WUSC) means that
there are numbers $\ua <2$, $\ut\ge 0$
and $\uC{\in [1,\infty)}$ such that
\begin{equation*}\label{eq:USC2}
 \psi(\lambda r )\le
\uC\lambda^{\,\ua} \psi( r )\quad \mbox{for}\quad \lambda\ge 1, \quad r >\ut.
\end{equation*}
In short, $\psi\in\WUSC{\ua}{ \ut}{\uC}$ or $\psi\in{\rm WUSC}$.
{\it Global} WUSC is WUSC$(\ua,0)$, etc. 

We call $\la$, $\lt$, $\lC$, $\ua$, $\ut$, $\uC$ the scaling characteristics of $\psi$. 
As pointed out in \cite[Remark 1.4]{MR3249349}, by inflating $\lC$ and $\uC$ we can replace $\lt$ with $\lt/2$ and $\ut$ by $\ut/2$ in the scalings, therefore we can always choose the same, arbitrarily small value $\theta=\lt=\ut>0$ in both local scalings WLSC and WUSC, if they hold at all.  The scalings characterize the so-called common bounds for $p_t(x)$ \cite[Theorem~21 and Theorem~26]{MR3165234}, and so they are 
natural conditions on $\psi$ in the unimodal setting.
The reader may also find  in \cite{MR3165234} many examples of L\'evy-Khintchine exponents which satisfy WLSC or WUSC. For instance $\psi(\xi)=|\xi|^{\alpha}$, the L\'evy-Khintchine exponent of the isotropic $\alpha$-stable L\'evy process in $\Rd$ with $\alpha\in (0,2)$, satisfies
$\WLSC{\alpha}{0}{1}$ and $\WUSC{\alpha}{ 0}{1}$.
The characteristic exponent $\psi(\xi)=(1+|\xi|^2)^{\alpha/2}-1$ of the relativistic $\alpha$-stable L\'evy process with $\alpha\in (0,2)$ satisfies  
WLSC$({\alpha},0)$
and WUSC$({\alpha},{ 1})$.
Other examples include $\psi(\xi)=|\xi|^{\alpha_1}+|\xi|^{\alpha_2}\in \WLSC{\alpha_1}{0}{1}\cap \WUSC{\alpha_2}{0}{1}$, where $0<\alpha_1<\alpha_2<2$, etc.
If $\psi(r)$ is $\alpha$-regularly varying at infinity and $0<\alpha<2$, then $\psi\in$WLSC($\la$)$\cap$WUSC($\ua$), with any $0<\la<\alpha<\ua<2$.
The connection of the scalings to the so-called Matuszewska indices of $\psi(r)$ is explained in \cite[Remark~2 and Section 4]{MR3165234}. 

If $\psi\in$WLSC$(\ualpha,\theta)$, then by \eqref{cVh1pg} (or see \cite[(1.8)]{MR3249349}) we get the following scaling at $0$:
\begin{align}\label{eq:scaling}
V(\varepsilon r)\le C \varepsilon^{\ualpha/2} V(r), \qquad 0<\varepsilon\le 1, \quad 0<r<1/\theta.
\end{align}
Here the range is $0<r<\infty$ if the lower scaling of $\psi$ is global, in agreement with \eqref{eq:scaling} and the convention $1/0=\infty$.
If $\psi\in$WUSC$(\ua,\theta)$, then, similarly,
\begin{align}\label{eq:lowscaling}
V(\varepsilon r)\ge C \varepsilon^{\oalpha/2} V(r), \qquad 0<\varepsilon\le 1, \quad 0<r<1/\theta.
\end{align}
We shall need $V^{-1}$, the inverse function of $V$ on $[0,\infty)$.
We 
let 
\begin{equation}\label{eq:defT}
  T(t)=V^{-1}(\sqrt{t}), \qquad t\geq 0.
\end{equation}
Put differently, $[V(T(t))]^2=t$.
For instance, 
$T(t)=t^{1/\alpha}$ for the isotropic $\alpha$-stable L\'evy process.
The functions $V$ and $T$ allow us to handle intrinsic difficulties which hampered  extensions of \cite{MR880981, MR2438694, MR3111870,MR3269724}
to general unimodal L\'evy processes, namely the lack of explicit formulas and estimates for the involved potential-theoretic objects.

We note that $T(t)<a$ if and only if $t<V^2(a)$, wherever $a,t\ge 0$.
The scaling properties of $T$ at zero reflect those of $\psi$ (at infinity) as follows.
\begin{lemma}\label{lem:inversescale}
If \eqref{eq:scaling} holds, $0< \varepsilon\le1$ and $0\le t< V(1/\theta)^2$, then
$T(\varepsilon t)\ge c \varepsilon^{1/\la} T(t)$.
\noindent
If 
\eqref{eq:lowscaling} holds, $0< \varepsilon\le1$ and $0\le t< V(1/\theta)^2$, then
$T(\varepsilon t)\le c \varepsilon^{1/\ua} T(t)$.
\end{lemma}
\begin{proof}
To prove the first assertion we note that
$T$ is increasing. If $0< t< V(1/\theta)^2$, and $0\le \varepsilon\le 1$, then 
$T(t)<1/\theta$ and $T(\varepsilon t)/T(t)\le 1$.
By \eqref{eq:scaling},
\begin{align*}
    \sqrt{\varepsilon}=\frac{V(T(\varepsilon t))}{V(T(t))}\le C \left( \frac{T(\varepsilon t)}{T(t)} \right)^{\la/2},
\end{align*}
as needed.
The proof of the second inequality is analogous but uses \eqref{eq:lowscaling}.
\end{proof}

By \eqref{eq:sublinear} and the proof of \autoref{lem:inversescale} we always have
\begin{align}\label{eq:Tscaling}
  T(\varepsilon t)\le c\sqrt{\varepsilon}T(t),\qquad 0<\varepsilon\le 1, \quad 0<r<\infty.
\end{align}

In what follows we always assume that $\nu$ is an infinite unimodal L\'evy measure on $\Rd$ with $d\ge 2$ and the L\'evy-Khintchine exponent defined by \autoref{eq:LKe} satisfies
$$
\psi \in {\rm WLSC}(\la,\theta)\cap {\rm WUSC}(\ua,\theta),
$$
where $0<\la\le \ua<2$, and $\theta\ge 0$. 
Many partial results below need less assumptions but for simplicity of presentation
we leave such observations to the interested reader.
\begin{definition}
We say that
$\A$ holds if for every $r>0$  there is $H_r\geq 1$ such that
\begin{equation*}\label{HR}
V(z)-V(y)\le H_r \,V^\prime(x)(z-y)\quad \text{whenever}\quad 0<x\le y\le z\le5x\leq5r.
\end{equation*}
We say that  $\As$ holds if
$H_\infty:=\sup_{r>0}H_r<\infty$.
\end{definition}
We may and do chose  $H_r$
nondecreasing in $r$.
By \cite[Section 7.1]{BGRptrf}, 
$\A$ always holds in our setting
because $\psi$ satisfies WLSC and WUSC. 
If  
$\psi\in$WLSC$({\la},{0})\cap$WUSC$({\ua},{0})$,
then $\As$ even holds.

\subsection{Heat kernel}
By \cite[Lemma~1.3]{MR3249349},
there is a $\Cf=\Cf(d)$ such that
\begin{equation}\label{B}p_t(x)\le {\Cf}\frac{t}{ |x|^dV^2(|x|)},\qquad t>0,\; x\in\Rdz,
\end{equation}
hence
\cite[(15)]{MR3165234},
\begin{align}\label{eq:levybound}
  \nu(x)\le \Cf  \frac{1}{V^2(|x|)|x|^d},\qquad x\neq 0.
\end{align}
Since $\psi\in$WLSC$(\la,\theta)$,
by \cite[Lemma 1.5]{MR3249349}
we have
  \begin{align}\label{eq:pttime}
    p_t(x)\le c T^{-d}(t), \qquad t<V^2(\theta^{-1}),\qquad x\in \Rd.
  \end{align}

We now discuss the heat remainder and the heat kernel of open sets $D\subset \Rd$.
As usual, 
$0\le r_D(t,x,y)\le p_t(x-y)$.
Indeed, one directly checks that  $[0,t)\ni s\mapsto Y_s=p(t-s,X_s,y)$ is a $\P_x$-martingale for each $x,y\in \Rd$. The martingale almost surely converges to $0$ as $s\to t$, and we let $Y_t=0$. By optional stopping, quasi-left continuity of $X$ and Fatou's lemma, for every stopping time $T\le t$ we have $\E_x Y_T\le \E_x Y_0=p(t,x,y)$. The inequality
$r_D(t,x,y)\le p_t(x-y)$
follows by taking $T=\tau_D\wedge t$. The next result is a consequence of the strong Markov property of $X_t$.
\begin{lemma}\label{rDrelative}
Consider open sets $D\subset F\subset \Rd$. For all $t>0$ and $x,y\in \Rd$,
  \begin{align*}
    p_F(t,x,y)-p_D(t,x,y)=\E^y \left[\tau_D<t,X(\tau_D)\in F\setminus D;\ p_F(t-\tau_D,X(\tau_D),x)\right].
  \end{align*}
\end{lemma}
\begin{proof}
We repeat verbatim the proof of \cite[Proposition 2.3]{MR2438694}.
\end{proof}
Here is a well-known Ikeda-Watanabe formula for the joint distribution of $X(\tau_D)$ and $\tau_D$, 
see \cite[Proposition~2.5]{MR2231884} 
or \cite[(27)]{2014arXiv1411.7952B} for proof.
\begin{lemma}\label{propIW}
 Let $D\subset \Rd$ be open. For $x\in D$, $t_2\ge t_1\ge 0$ and $A\subset (\overline{D})^c$,  
\begin{align*}
  \P^x(X(\tau_D)\in A, t_1<\tau_D<t_2)=\int_D \int_{t_1}^{t_2} p_D(s,x,y)ds \int_A \nu(y-z)dzdy.
\end{align*}
\end{lemma}

We denote
$\delta_D(x):={\rm dist}(x,D^c)$, $x\in \Rd$.
\begin{lemma}\label{lem:rDbound}
We have
\begin{align}\label{eq:rDtime}
  r_D(t,x,y)&\le C T(t)^{-d},
\end{align}
and
\begin{align}\label{eq:zor}
r_D(t,x,y)&\le \Cf\frac{t}{V^2(\delta_D(x))\delta_D^{d}(x)}, \qquad x,y\in \Rd.
\end{align}
\end{lemma}
\begin{proof}
Since $\psi \in \WLSC{\la}{\theta}{\lC}$, we have 
\eqref{eq:pttime}, which yields \eqref{eq:rDtime}.
By \eqref{eq:rdE}, \eqref{B}, and symmetry, 
  \begin{align*}
    r_D(t,x,y)=r_D(t,y,x)\le \E^y\left[\tau_D<t;\; \Cf\frac{t-\tau_D}{V^2(|X(\tau_D)-x|)|X(\tau_D)-x|^d}\right].
  \end{align*}
Since $|X(\tau_D)-x|\le \delta_D(x)$ and $V$ is increasing, we obtain \eqref{eq:zor}.
\end{proof}

Recall that $\H$ is a half-space and $C_\H(t)$ is defined immediately before \autoref{mainthm}.
\begin{lemma}\label{lem:halfspace} 
  If 
$T(t)<1/\theta$, then 
    $C_\H(t)\le cT(t)^{-d+1}$.

\end{lemma}
\begin{proof}
Denote $r(t,q)=r_\H(t,(q,0,\cdots,0),(q,0,\cdots,0))$.
By \autoref{eq:zor} and \eqref{eq:scaling},
  \begin{align*}
    \int_{T(t)}^\infty r(t,q)dq \le c \int_{T(t)}^\infty \frac{V^2(T(t))}{ V^2(q)q^d} dq
    \le c \int_{T(t)}^\infty \frac{T(t)^{\ualpha}}{q^{d+\ualpha}} dq=cT(t)^{1-d}.
  \end{align*}
Using \eqref{eq:rDtime} we get
  \begin{align*}
    \int_0^{T(t)} r(t,q)dq \le c \int_0^{T(t)} T(t)^{-d} dq = c T(t)^{1-d}.
  \end{align*}
\end{proof}
To obtain a lower bound for $C_\H(t)$ we shall use the existing heat kernel estimates for geometrically regular domains.
Recall that open set $D\subset \Rd$ satisfies the inner (outer) ball condition at scale $R>0$ if 
for every $Q\in \partial D$ there is a ball $B(x\rq{},R)\subset D$ (a ball $B(x\rq{}\rq{},R)\subset D^c$) such that $Q\in \partial B(x\rq{},R)$ ($Q\in \partial B(x\rq{}\rq{},R)$, respectively). An open set $D$ is $R$-smooth if it satisfies both the inner and the outer ball conditions at some scale $R>0$. 
We call $B(x\rq{},R)$ and $B(x\rq{}\rq{},R)$  the inner ball and the outer ball, respectively.

In the next lemma we collect a number of results from \cite{MR3249349}. 
For brevity in what follows we sometimes write $T=T(t)$, where $t>0$ is given.
\begin{lemma}\label{pDbound}
Let open $D\subset \Rd$ satisfy the outer ball condition at scale $R<1/\theta$.
 There is a constant $c$ such that for $T\vee |x-y|<1/\theta$,
      \begin{align*}
	p_{D}(t,x,y)&\le c\left(\frac{V(\delta_D(x))}{V(T\wedge R)}\wedge 1\right)\left(\frac{V(\delta_D(y))}{V(T\wedge R)}\wedge 1\right)\left( T^{-d}\wedge \frac{V^2(T)}{|x-y|^dV^2(|x-y|)}\right).
  \end{align*}
\end{lemma}
\begin{proof}
We have $\A$.
We note that $\sqrt{t}=V(T)$ and use the second part of \cite[Corollary 2.4]{MR3249349}. 
We need to justify that the quotient $H_R/J^4(R)$ is bounded, where $H_R$ is the constant from $\A$ and $J(R)=\inf_{0<r\le R}\nu(B(0,r)^c)V^2(r)$. 
To this end we observe that $H_R$ is increasing, and $J(R)$ is nonincreasing, hence we get an upper bound for this quotient by replacing $R$ with $1/\theta$. If $\theta=0$, which we also allow,
then by \cite[Proposition 5.2, Lemma 7.2 and 7.3]{BGRptrf} the quotient is bounded as a function of $R$.
By \cite[Lemma 1.6]{MR3249349} with $r=1/2$, we also have $p_{t/2}(0)\le cT^{-d}(t)$.
\end{proof}
\begin{lemma}\label{lem:lbCH}
We have
$C_\H(t)\approx T(t)^{-d+1}\approx p_t(0)T(t)$ as $t\to 0$.
\end{lemma}
\begin{proof}
By \autoref{pDbound} and \autoref{eq:wH} there is $\varepsilon>0$ such that 
$r(t,q)\ge \frac 12 p_t(0)$ if $V(q)<\varepsilon \sqrt{t}$.
Since $\psi\in$WUSC, by scaling of $V$ there is $c>0$ such that for $0<q\le c T(t)$ the condition is satisfied and we have
  \begin{align*}
    \int_0^{cT(t)} r(t,q)dq \ge \frac 12 
\int_0^{cT(t)} T(t)^{-d} dq = \frac c2 T(t)^{1-d}.
  \end{align*}
By WUSC and WLSC we have $p_t(0)\approx T(t)^{-d}$, see \cite[(23)]{MR3165234}.
\end{proof}

\subsection{Green function}
For $M\ge 0$, the truncated Green function of $D$ is defined as
  \begin{align*}
    G_D^M(x,y)=\int_0^M p_D(t,x,y)dt, \qquad x,y\in \Rd.
  \end{align*}
The Green function of $D$ is 
  \begin{align*}
 G_D(x,y)=\int_0^\infty p_D(t,x,y)dt=G_D^\infty(x,y).
  \end{align*}

\begin{lemma}
Let open $D\subset \Rd$ satisfy the outer ball condition at scale $R<1/\theta$,
$x,y\in \Rd$ and $|x-y|<1/\theta$. Let $M=V^2(R)$. Then
  \begin{align}\label{eq:green2}
    G_D^{M} (x,y)\le c \frac{V(\delta_D(y))V(\delta_D(x))}{|x-y|^d},
  \end{align}
and
  \begin{align}\label{eq:green1}
    G_D^{M} (x,y)\le c \frac{V(\delta_D(y))V(|x-y|)}{|x-y|^d}.
  \end{align}
Furthermore, if $d>2$ or {\rm WUSC}$(\ua,0)$ holds,
 then \eqref{eq:green2} and \eqref{eq:green1} even hold for $M=V^2(1/\theta)$, including the case
of global {\rm WLSC} ($M=\infty$).
\end{lemma}
\begin{proof}
Assuming $T<R\wedge |x-y|$,   by Lemma~\ref{pDbound} we get
  \begin{align*}
    p_D(t,x,y)\le c V(\delta_D(y)) \frac{V(T\wedge \delta_D(x))}{V^2(|x-y|)|x-y|^d},
  \end{align*}
hence
  \begin{align*}
    \int_0^{V^2(|x-y|\wedge R)} p_D(t,x,y) dt&\le  c\frac{V(\delta_D(x))}{V^2(|x-y|)|x-y|^d}\int_0^{V^2(|x-y|\wedge R)} V(T\wedge \delta_D(x)) dt
    \\&\le 
    c\frac{V(\delta_D(x))V^2(|x-y|\wedge R)V(|x-y|\wedge \delta_D(x))}{|x-y|^dV^2(|x-z|)}
    \\&\le 
    c\frac{V(\delta_D(x))V(|x-y|\wedge \delta_D(x))}{|x-y|^d}.
  \end{align*}
  This establishes \eqref{eq:green1} and \eqref{eq:green2} for small times. 
Then,
  \begin{align*}
    \int_{V^2(|x-y|)}^{V^2(R)} p_D(t,x,y) dt\le c V(\delta_D(x))  \int_{V^2(|x-y|)}^{V^2(R)} \frac{T^{-d}(t)}{\sqrt{t}} dt.
  \end{align*}
By WUSC and  \autoref{lem:inversescale},
  \begin{align*}
    \frac1{T(t)}\le \frac{c \varepsilon^{1/\oalpha}}{T(\varepsilon t)}.
  \end{align*}
With this in mind we obtain
  \begin{align*}
      \int_{V^2(|x-y|)}^{V^2(R)} \frac{T^{-d}(t)}{\sqrt{t}} dt&\le c\int_{V^2(|x-y|)}^\infty \frac{V^{2d/\oalpha}(|x-y|)}{t^{d/\oalpha+1/2}T^{d}(V^2(|x-y|))} dt
    \\&=c \frac{V^{d/\oalpha}(|x-y|)}{|x-y|^d} \left[V^2(|x-y|)\right]^{-d/\oalpha-1/2+1},
  \end{align*}
where the integral converges, because
$d/\oalpha+1/2>1$ (recall that $\oalpha<2$). 
We thus get \eqref{eq:green1}.
To finish the proof of \eqref{eq:green2} we note that
  \begin{align*}
    \int_{V^2(|x-y|)}^{V^2(R)} p_D(t,x,y) dt\le c V(\delta_D(x))V(\delta_D(y)) \int_{V^2(|x-y|)}^{V^2(R)} \frac{T^{-d}(t)}{t} dt,
  \end{align*}
 and we proceed as before.
\end{proof}
\subsection{Poisson kernel}
For $M\ge 0$,  the truncated Poisson kernel is defined as
  \begin{align*}
    K_D^{M} (x,z)=\int_D G_D^M(x,y)\nu(y-z)dy, \qquad x\in D,\ z\in D^c.
  \end{align*}

\begin{lemma}\label{lem:ojP} Let open $D\subset \Rd$ satisfy the outer ball condition at scale $R$. 
If $\diam (D\cup\{z\})<1/\theta$, then 
  \begin{align*}
    K_D^{V(R^2)/2}(x,z)\le \frac{V(\delta_D(x))}{V(\delta_D(z))}\frac{c}{|x-z|^d}, \qquad x\in D,\  z\in D^c.
  \end{align*}
\end{lemma}
\begin{proof}
The previous lemma gives an estimate for $G_D^{V^2(R)}$,
and the L\'evy measure is controlled by \eqref{eq:levybound}. Thus,
  \begin{align*}
  K_D^{V(R^2)/2}(x,z)\le cV(\delta_D(x)) \int_D \frac{V(|x-y|)\wedge V(\delta_D(y))}{|x-y|^d|y-z|^d V^2(|y-z|)} dy.
  \end{align*}
  Note that $|x-y|\ge |x-z|/2$ or $|y-z|\ge |x-z|/2$. Furthermore, if $|x-y|\ge |y-z|$, then $|x-y|\ge |x-z|/2$. Therefore, it is enough to verify that
  \begin{align*}
    I&:=\int_D \frac{V(\delta_D(y))}{|y-z|^d V^2(|y-z|)}dy\le \frac{C}{V(\delta_D(z))},\ \mbox{ and}\\
    I\!I&:=\int_{D\cap \{|x-y|<|y-z|\}} \frac{V(|x-y|)}{|x-y|^d V^2(|y-z|)}dy\le \frac{C}{V(\delta_D(z))}.
  \end{align*}
Considering $I$ we  note that $\delta_D(y)\le |y-z|$, hence
  \begin{align*}
    I\le \int_{|y-z|>\delta_D(z)} \frac{|y-z|^{-d}}{V(|y-z|)}dy 
\le c \int_{\delta_D(z)}^{1/\theta} \frac{dr}{rV(r)}
  \end{align*}
  Using the scaling \eqref{eq:scaling} we get
  \begin{align*}
    I\le \frac{c}{V(\delta_D(z))}\int_{\delta_D(z)}^\infty \left( \frac{\delta_D(z)}{r} \right)^{\ualpha/2}\frac{dr}{r}=\frac{c}{V(\delta_D(z))}.
  \end{align*}
 To verify the estimate for $I\!I$ we also use the scaling properties of $V$. For $y\in D$ we have $|y-z|<1/\theta$, hence
  \begin{align*}
    I\!I
    &\le c \int_{|x-y|\le |y-z|} \left(\frac{|x-y|}{|y-z|}\right)^{\ualpha/2}\frac{dy}{|x-y|^d V(|y-z|)}
    \\&\le \frac{c}{V(\delta_D(z))} \int_0^{|y-z|} \left(\frac{r}{|y-z|}\right)^{\ualpha/2} \frac{dr}{r}
= \frac{c}{V(\delta_D(z))} \frac{2}{\ualpha}.
\end{align*}
\end{proof}

\section{Proof of the main result
}\label{s:mt}
For the convenience of the reader in the following statement 
we repeat our standing assumptions; see also the definition of $V$ in \autoref{sec:uni} and that of $T$ in \autoref{eq:defT}.
\begin{theorem}\label{th:mainR}
Let $\nu$ be an infinite unimodal L\'evy measure on $\Rd$ with $d\ge 2$, and let the L\'evy-Khintchine exponent 
\autoref{eq:LKe} satisfy
$\psi\in {\rm WLSC}(\la,\theta)\cap {\rm WUSC}(\ua,\theta)$, where $0<\la\le \ua<2$ and $\theta\ge 0$. Let open bounded set $D\subset \Rd$ be $R$-smooth with $0<R<1/\theta$. There is a constant $c_\theta$ depending only on $\nu$ and $\theta$ such that  if $0<t<V^2(\theta^{-1})$, or $T(t)<1/\theta$, then the trace \autoref{eq:deftr} of the Dirichlet heat kernel 
\autoref{eq:wH} satisfies
  \begin{align}\label{eq:result}
    \Big| \tr(t,D) - |D|p_t(0)+|\partial D|C_\H(t) \Big|\le c_\theta |D|p_t(0) \frac{T(t)^2}{R^2}.
  \end{align}
If $\theta=0$, then \autoref{eq:result} holds for all $t>0$.
\end{theorem}
Recall that \autoref{lem:lbCH} asserts that 
$C_\H(t)\approx p_t(0)T(t)$ and  $p_t(0)\approx T(t)^{-d}$ as $t\to 0$, so the approximation of the trace in \autoref{th:mainR} is given in terms of powers of $T(t)$. 
\begin{proof}[Proof of \autoref{mainthm}]
The result is a direct consequence of \autoref{eq:pttime}, \autoref{lem:lbCH} and 
\autoref{th:mainR}, where we take $\theta>0$ so small that $R<1/\theta$ (see \autoref{ss:s} in this connection). 
\end{proof}

In the course of the  proof of \autoref{th:mainR}, which now follows, we usually write $T=T(t)$. 
As mentioned in the Introduction,
\begin{align*}
  \tr(t,D)-|D|p_t(0)=\int_D p_D(t,x,x)dx - \int_D p(t,x,x)dx = - \int_D r_D(t,x,x) dx.
\end{align*}
We only need to show that 
\begin{align}\label{eq:rDmCH}
  \left|\int_{D} r_D(t,x,x)dx - |\partial D| C_\H(t)\right|\le \frac{cT^2}{T^d R^2}.
\end{align}

We first consider $T=T(t)\ge R/2$, and we have
\begin{align*}
\int_D r_D(t,x,x)\le \int_D p_t(0) dx \le |D|p_t(0) \le 4|D|p_t(0)\frac{T^2}{R^2}.
\end{align*}
By \autoref{lem:halfspace},
\begin{align*}
  |\partial D| C_\H(t)=|\partial D|\int_0^\infty r_\H(t,(q,0,\cdots,0),(q,0,\cdots,0)) dq\le \frac{c|D|}{R} T^{1-d}\le \frac{c|D|T^{2-d}}{R^2}.
\end{align*}
By 
\cite[(23)]{MR3165234}, we see that \autoref{eq:result} holds trivially in this case.

From now on we assume that  $T<R/2$.
For $r>0$ we let $D_{r}=\left\{ x\in D: \delta_D(x)>r \right\}$. We have $D=D_{R/2}\cup (D\setminus D_{R/2})$.
In analyzing the decomposition we shall often use our assumptions $R<1/\theta$ and $|x-y|<1/\theta$, and the heat kernel estimates from \autoref{pDbound}.
By \autoref{lem:rDbound},
\begin{align}\label{eq:inside}
  \int_{D_{R/2}} r_D(t,x,x)dx\le C|D_{R/2}|\frac{V^2(T)}{V^2(R/2)R^d}\le C|D|\frac{1}{R^2R^{d-2}}\le C|D|\frac{1}{R^2T^{d-2}}.
\end{align}
Thus, the integral 
gives insignificant contribution to the trace.

To handle the integration near $\partial D$, we shall estimate the heat remainder of $D$ using the heat remainder of halfspace. Let $x^*\in \partial D$ be such that $|x-x^*|=\delta_D(x)$. Let $I$ and $O$ be the (inner and outer) balls with radii $R$ such that $\partial I\cap \partial O=\{x^*\}$ and $I\subset D\subset O^c$. Let $\H(x)$ denote the halfspace satisfying $I\subset \H(x)\subset O^c$. By domain monotonicity of the heat remainder, and by \autoref{rDrelative},
\begin{align*}
  |r_D(t,x,x)-r_{\H(x)}(t,x,x)|&\le r_{I}(t,x,x)-r_{O^c}(t,x,x)
  \\&=p_{O^c}(t,x,x)- p_{I}(t,x,x)
  \\&=\E^x \left[\tau_I<t,X(\tau_I)\in O^c; p_{O^c}(t-\tau_I,X(\tau_I),x)\right].
\end{align*}
The next result is an analogue of \cite[Proposition 3.1]{MR2438694}.
\begin{proposition}\label{propjump} 
If $T<R/2$, then
  \begin{align*}
    \E^x &\left[\tau_I<t,X(\tau_I)\in O^c; p_{O^c}(t-\tau_I,X(\tau_I),x)\right]
    \le \frac{c}{R}\left( \frac{V(T)}{\delta_D(x)^{d-1}V(\delta_D(x))}\wedge T^{1-d} \right).
  \end{align*}
\end{proposition}
The proof of \autoref{propjump} is given in \autoref{proofprop}.

\begin{lemma}\label{lem:claim1}
 If $T<R/2$, then
  \begin{align}\label{eq:claim1}
  \left|\int_{D\setminus D_{R/2}} \!\!\!r_D(t,x,x)-r_{\H(x)}(t,x,x)\ dx\right|\le \frac{c|D|T^2}{R^2T^d}.
\end{align}
\end{lemma}
\begin{proof}
  This is an analog of 
\cite[Claim 2]{MR2438694} and is proved as follows. 
By the coarea formula and \autoref{propjump} we find that the left side of \eqref{eq:claim1} is bounded above by
\begin{align*}
  \frac{cT}{RT^d} \int_0^{R/2}|\partial D_q| \left( \frac{T^{d-1}V(T)}{q^{d-1}V(q)}\wedge 1 \right)dq.
\end{align*}
Therefore \cite[Corollary 2.14(i)]{MR2438694} gives a simplified bound
\begin{align*}
  \frac{c|\partial D|}{RT^{d-1}} \int_0^{R/2}\left( \frac{T^{d-1}V(T)}{q^{d-1}V(q)}\wedge 1 \right)dq.
\end{align*}
The integral over $(0,T)$ is clearly bounded by 
$T$. To estimate the integral from $T$ to $R/2$ we note that scaling \eqref{eq:scaling} for $q\in [T,R/2)$ yields $V(T)\le C(T/q)^{\ualpha/2} V(q)$. Also,
\begin{align*}
  \int_T^{R/2} q^{1-d-\ualpha/2} dq
  \le
  \int_T^\infty q^{1-d-\ualpha/2} dq<\infty,
\end{align*}
since $d+\ualpha/2>2$. 
\end{proof}
Recall that $r(t,q) = r_\H(t,(q,0,\cdots,0),(q,0,\cdots,0))$, and
$C_\H(t)=\int_0^\infty r(t,q)dq.$
\begin{lemma}\label{lem:claim2}
If $T<R/2$, then
  \begin{align}\label{eq:claim2}
    \left|\int_{D\setminus D_{R/2}} r_{\H(x)}(t,x,x)dx-|\partial D|\int_0^{R/2} r(t,q)dq\right|\le \frac{c|D|T^2}{R^2T^d}.
  \end{align}
\end{lemma}
\begin{proof}
Using the coarea formula we get
  \begin{align*}
    \int_{D\setminus D_{R/2}} r_\H(x)(t,x,x)dx = \int_{0}^{R/2} |\partial D_q| r(t,q)dq.
  \end{align*}
  Hence the left side of the inequality \eqref{eq:claim2} is bounded by
  \begin{align*}
    \int_0^{R/2} &\big||\partial D_q|-|\partial D|\big| r(t,q)dq
    \le
    \frac{C|D|}{R^2}\int_0^{R/2} q\ r(t,q)dq,
  \end{align*}
as follows from \cite[Corollary 2.14(iii)]{MR2438694}. 
For $q\in (0,T]$ we have $r(t,q)\le p_t(0)$, hence
  \begin{align*}
    \int_0^T q r(t,q)dq \le c \int_0^{T} \frac{q}{T^{d}} dq = c T^{2-d}.
  \end{align*}
For the remaining integration, using \autoref{eq:zor}
and \eqref{eq:scaling}, we get
  \begin{align*}
    \int_T^{1/\theta} q r(t,q)dq &\le c \int_T^{1/\theta} \frac{t}{q^{d-1}V^2(q)} dq\le c \int_T^{1/\theta}
\frac{V^2(T)}{q^{d-1}V^2(q)} dq
    \\&\le c\int_{T}^{1/\theta} \left(\frac{T}{q}\right)^{\ualpha}\frac{dq}{q^{d-1}}\le c T^{2-d} \int_1^\infty q^{-d+1-\ualpha}dq.
  \end{align*}
  The last integral converges since $d\ge 2$ and 
$\ualpha>0$.
\end{proof}
Thus, for $T<R/2$ we have by \autoref{lem:rDbound}
\begin{align*}
  |\partial D|\int_{R/2}^\infty r(t,q)dq &\le \frac{c|D|}{R} \int_{R/2}^\infty \frac{V^2(T)}{q^dV^2(q)}dq 
  \le \frac{c|D|}{R} \int_{R/2}^{\infty} \frac{dq}{T^{d-2}q^2}=\frac{CT^2}{R^2T^d},
\end{align*}
which is 
a lower order term.
By \autoref{lem:claim1}, \autoref{lem:claim2} and \eqref{eq:inside} we obtain \eqref{eq:rDmCH}.

\section{Proof of \autoref{propjump}.}\label{proofprop}
Let $x^*=0$, $a=(-R,0,\dots,0)$, $b=(R,0,\dots,0)$, $I=B(a,R)$ and $O=B(b,R)$. 
This also means that $x=(x_0,0,\dots,0)$ with $0\le x_0<R/2$, and $\delta_I(x)=|x|$, see \autoref{fig:balls}. Recall that $t<V^2(R/2)$ or equivalently $T<R/2$.
 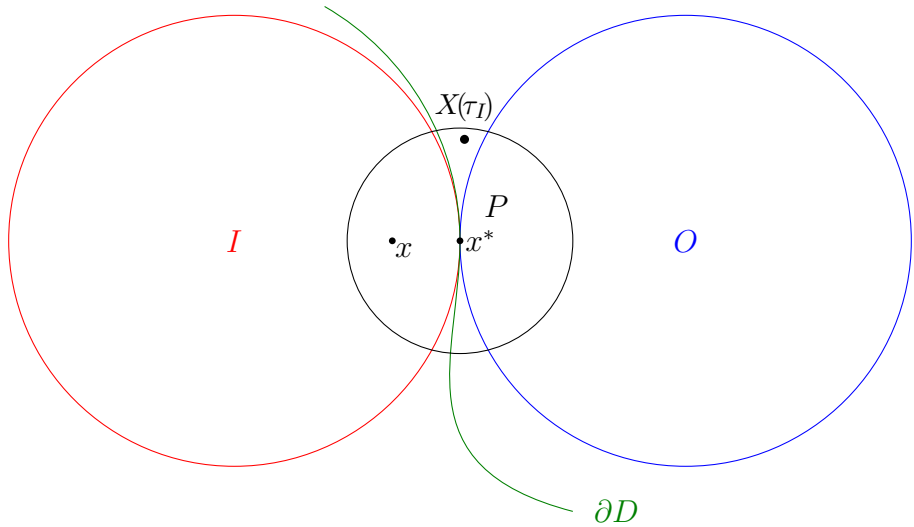
\begin{figure}[t]
   \begin{center}
     \begin{tikzpicture}[rotate=-90,scale=3]
       \draw[red] (0,-1) circle (1) node {$I$};
       \draw[blue] (0,1) circle (1) node {$O$};
       \draw[green!50!black] (0,0) .. controls (0.5,0) and (1,-0.27) .. (1.2,0.5) node [right] { \ $\partial D$};
       \draw[green!50!black] (0,0) arc (90:150:1.2);
       \draw (0,0) circle (0.5) node [above right=5pt] {$P$};
       \fill (0,-0.3) circle (0.015) node [below right=-3pt] {$x$};
       \fill (0,0) circle (0.015) node [right=-2pt] {$x^*$};
       \fill (-0.45,0.020) circle (0.02) node  [above=2.8pt]{\small $X\!(\!\tau_I\!)$};
     \end{tikzpicture}
   \end{center}
   \caption{Balls $I\subset D$ (left), $O\subset D^c$ (right) and $P$ (middle), and ``a~short jump'' to point $X(\tau_I)$. Here $x\in P$ and $|x|=\delta_I(x)$.}
   \label{fig:balls}
 \end{figure}
Before we proceed to the heart of the matter
we need the following lemma based on spherical integration developed in \cite[Pages 355--355]{MR1490808} and later used in \cite{MR2438694,MR2568694}.

\begin{lemma}\label{lem:spherical}
For $s<R$ we have
  \begin{align}\label{eq:spherical}
    \int_{(O^c\setminus I)\cap B(0,s)} \frac{dz}{|x-z|^\beta}\frac{V(\delta_{O^c}(z))}{V(\delta_I(z))}\le c
    \begin{cases}
      |x|^{d+1-\beta}/R&\text{ if }\ \beta>d+1,\\
      s^{d+1-\beta}/R&\text{ if }\ \beta<d+1.
    \end{cases}
  \end{align}
\end{lemma}
\begin{proof}
  First we consider $V(x)=x^{\alpha/2}$ with $\alpha\in[0,2)$. 
  Let $z\in A=(O^c\setminus I)\cap B(0,s)$. Note that $|x-z|\ge |x|$. 
If $|x-z|\le 2|x|$, then $|z|\le |x-z|+|x|\le 3|x|$, which leads to the integral
 \begin{align*}
   \int_{A \cap \{|x-z|\le 2|x|\}} &\frac{dz}{|x-z|^\beta}\frac{\delta_{O^c}^{\alpha}(z)}{\delta_I^{\alpha}(z)}
   \le
   \frac{1}{|x\wedge s|^\beta}\int_{A\cap \{|z|\le 3(|x|\wedge s)\}} \frac{\delta_{O^c}^{\alpha}(z)}{\delta_I^{\alpha}(z)}dz.
 \end{align*}
 The last integral is similar to \cite[(3.21)]{MR2438694}. Using \cite[(3.23) and (3.24)]{MR2438694} we get the following upper bound
 \begin{align*}
   \frac{c}{|x\wedge s|^\beta}\int_0^{3(|x|\wedge s)} \frac{r^d}{R} dr=\frac{c(|x|\wedge s)^{d+1-\beta}}{R}.
 \end{align*}
If $|x-z|\ge 2|x|$, then $|x-z|\ge |z|/2$ and $|z|\ge |z-x|-|x|\ge |x|$. By \cite[(3.24)]{MR2438694},
 \begin{align*}
   \int_{A \cap \{|x-z|> 2|x|\}} &\frac{dz}{|x-z|^\beta}\frac{\delta_{O^c}^{\alpha}(z)}{\delta_I^{\alpha}(z)}
   \le
   c\int_{A\cap \{s\ge |z|\ge |x|\}} \frac{1}{|z|^\beta}\frac{\delta_{O^c}^{\alpha}(z)}{\delta_I^{\alpha}(z)}dz
   \le 
   \frac{c}R\int_{|x|\wedge s}^{s} r^{d-\beta} dr.
 \end{align*}
 If $\beta>d+1$, then the last integral is bounded by $c|x|^{d+1-\beta}$, while for $\beta<d+1$ we get the upper bound $cs^{d+1-\beta}$.

This settles \eqref{eq:spherical} for $V(x)=x^{\alpha/2}$ with $\alpha\in[0,2)$. Note that 
the form of the right hand side of \eqref{eq:spherical} does not depend on $\alpha$.

Consider general $\psi\in$WUSC$(\ua)$ and the corresponding ladder-height function $V$. Due to the scaling property \eqref{eq:lowscaling} we have
 \begin{align*}
   \frac{V(\delta_{O^c}(z))}{V(\delta_I(z))}\le 
   c\frac{\delta_{O^c}^{\oalpha}(z)}{\delta_I^{\oalpha}(z)},\quad \text{ if }\delta_{O_c}(z)\ge \delta_I(z).
 \end{align*}
If $\delta_{O^c}(z)\le \delta_{I}(z)$, then the fraction is bounded by $1$, since $V$ is monotone.
Therefore, we can use the previous special case with $\alpha=\oalpha$ and $\alpha=0$ to finish the proof.
\end{proof}

We return to the core proof of \autoref{propjump}.
In view of \autoref{propIW} we want to estimate
  \begin{align*}
    \E^x &\left[\tau_I<t,X(\tau_I)\in O^c; p_{O^c}(t-\tau_I,X(\tau_I),x)\right]\nonumber
    \\&\qquad\qquad=
    \int_I\int_0^t p_I(s,x,y)\int_{O^c\setminus I} \nu(y-z) p_{O^c}(t-s,x,z)\,dzdsdy
    \\&\qquad\qquad=
    I_1+I_2+I_3,
  \end{align*}
which  splits the integration into
three subregions, as specified and estimated below:
 \begin{align*}
   I_1:&\qquad |z|>R/2,\\ 
   I_2:&\qquad t/2<s<t\text{ AND }|x-z|<T\text{ AND }|z|\le R/2,\\
   I_3:&\qquad (s<t/2\text{ OR }|x-z|>T)\text{ AND }|z|\le R/2.
 \end{align*}
The setting, especially that of $I_2$, is illustrated on \autoref{fig:balls}.
\subsection{Long jump: integral $I_1$}

On $I_1$ we have $|z|>R/2$, hence $|x-z|\ge R/3$, thus by \autoref{B}
\begin{align*}
  I_1&= \int_I \int_0^t p_I(s,x,y) \int_{|z|>R/2} \nu(y-z) p(t-s,z,x)dsdzdy 
  \\&\le
  \frac{ct}{R^d V^2(R/3)}\int_I \int_0^t p_I(s,x,y) \int_{P^c} \nu(y-z) dsdzdy
  \\&=
  \frac{ct}{R^d V^2(R/3)} \P^x(\tau_I<t,|X(\tau_I)|>R/2)\le \frac{cV^2(T)}{R^d V^2(R/2)},
\end{align*} 
where the last inequality follows from sublinearity \eqref{eq:sublinear} of $V$.
Since $T<R/2$, we have
\begin{align*}
  \frac{cV^2(T)}{R^d V^2(R/2)}\le \frac{c}{R^d}\le \frac{c}{RT^{d-1}}.
\end{align*}
Since $|x|<R/2$, by monotonicity of $V$ we get
\begin{align*}
  \frac{cV^2(T)}{R^d V^2(R/2)}\le \frac{cV(T)}{R^dV(R/2)}\le \frac{c V(T)}{R|x|^{d-1}V(|x|)}.
\end{align*}

\subsection{Long exit time and short jump: integral $I_2$.}

Here we have $|x|\le |x-z|<T$, and $|z|\le |x-z|+|x|< 2T$.
By \autoref{pDbound}, $t/2<q<T$ and \eqref{eq:Tscaling},
\begin{align*}
  p_I(q,x,y)\le T^{-d} \frac{V(\delta_I(y))}{V(T)}.
\end{align*}
Let $S=(O^c\setminus I)\cap \{|z|<2T\}$. We get the following upper bound,
\begin{align*}
  I_2
  &=\int_I \int_{t/2}^t p_I(q,x,y) \int_S \nu(y-z) p_{O^c}(t-q,z,x)dqdzdy 
  \\&\le
  c\int_I T(t)^{-d} \frac{V(\delta_I(y))}{V(T)} \int_S \frac{1}{|y-z|^d V^2(|y-z|)}G_{O^c}^{V^2(R/2)}(x,z)dz dy
  \\&\le
  \frac{cT^{-d}}{V(T)} 
  \int_S \int_I \frac{V(\delta_I(z))}{|y-z|^d V(|y-z|)}\frac{G_{O^c}^{V^2(R/2)}(x,z)}{V(\delta_I(z))}dy dz,
\end{align*}
where we use $\delta_I(y)\le |y-z|$. 
Scaling \eqref{eq:scaling} gives
\begin{align*}
  I_2\le 
  \frac{cT^{-d}}{V(T)}\int_S \int_{B^c(z,\delta_I(z))} \frac{\delta_I^{\ualpha/2}(z)}{|y-z|^{d+\ualpha/2}}\frac{G_{O^c}^{V^2(R/2)}(x,z)}{V(\delta_I(z))}dy dz.
\end{align*}
We then rewrite the inner integral in spherical coordinates, use Green function estimate \eqref{eq:green2} and $|x|<T$,
\begin{align}
  I_2&\le 
  \frac{cT^{-d}}{V(T)} \int_{\delta_I(z)}^{\infty} \frac{\delta_I^{\ualpha/2}(z)dr}{r^{1+\ualpha/2}}
  \int_S\frac{V(|x|)V(\delta_{O^c}(z))}{|x-z|^d V(\delta_I(z))}dz\nonumber
  \\&\le
  cT^{-d} \int_1^{\infty} \frac{dr}{r^{1+\ualpha/2}}
  \int_S\frac{V(\delta_{O^c}(z))}{|x-z|^d V(\delta_I(z))}dz
  =
  cT^{-d} \int_S\frac{V(\delta_{O^c}(z))}{|x-z|^d V(\delta_I(z))}dz.\label{eq:I2bound}
\end{align}
Using \autoref{lem:spherical} with $\beta=d$ and $s=2T$ we get
\begin{align*}
  I_2\le \frac{cT^{1-d}}{R}.
\end{align*}
Since $|x|<T$, we get the desired estimate from \autoref{propjump}.

\subsection{Short exit time or medium jump: integral $I_3$.}
Let $S=(O^c\setminus I)\cap \{|z|<R/2\}$. We have $|x-z|>T$ or $s<t/2$. In either case, \autoref{pDbound} and sublinearity of $V$ implies
\begin{align*}
  p_{O^c}(t-s,x,z)\le \left( T^{-d}\wedge \frac{V^2(T)}{|x-z|^d V^2(|x-z|)} \right)\frac{V(\delta_{O_c}(z))}{V(T)}.
\end{align*}
Therefore by \autoref{lem:ojP},
\begin{align*}
  I_3&
\le
\int_I \int_0^{V^2(R/2)} \!\!\!\!\!\!p_I(s,x,y)\int_S \nu(y-z) \left( T^{-d}\wedge \frac{V^2(T)}{|x-z|^d V^2(|x-z|)} \right)\frac{V(\delta_{O_c}(z))}{V(T)} dzdsdy
  \\&=
  c\int_S K_I^{V^2(R)}(x,z)  \left( T^{-d}\wedge \frac{V^2(T)}{|x-z|^d V^2(|x-z|)} \right)\frac{V(\delta_{O_c}(z))}{V(T)}dz
  \\&\le
  c\int_S \frac{V(|x|)}{V(\delta_{I}(z))}\frac{1}{|x-z|^d}    \left( T^{-d}\wedge \frac{V^2(T)}{|x-z|^d V^2(|x-z|)} \right)\frac{V(\delta_{O_c}(z))}{V(T)}dz.
\end{align*}
If $|x-z|<T$, then we are satisfied with  $T^{-d}$ from the minimum and we note $V(|x|)<V(T)$.
We arrive at \eqref{eq:I2bound}, and finish the proof in the same way as in the previous cases.

We are left with the case $|x-z|>T$, and we have
\begin{align*}
  I_3&\le 
  c V(T)\int_S \frac{V(|x|)}{|x-z|^{2d} V^2(|x-z|)}\frac{V(\delta_{O_c}(z))}{V(\delta_{I}(z))}dz.
\end{align*}
Since $\psi\in$WLSC$(\la)$, we get
\begin{align*}
  I_3&\le
  cV(T)\int_S \frac{|x|^{\ualpha/2}}{|x-z|^{2d+\ualpha/2}V(|x-z|)}\frac{V(\delta_{O_c}(z))}{V(\delta_{I}(z))}dz
  \\&\le
  \frac{cV(T)|x|^{\ualpha/2}}{(T\vee |x|)^{d-1}V(T\vee |x|)}\int_S \frac{V(\delta_{O_c}(z))}{|x-z|^{d+1+\ualpha/2}V(\delta_{I}(z))}dz,
\end{align*}
where the last inequality follows from the monotonicity of $V$, since $|x-z|\ge |x|\vee T$.
Now we use \autoref{lem:spherical} with $\beta=d+1+\ualpha/2$, to get
\begin{align*}
  I_3\le \frac{cV(T)}{(T\vee |x|)^{d-1}V(T\vee |x|)R}.
\end{align*}
Here the right hand side  is comparable with the required upper bound.


\end{document}